\documentclass{article}    

\usepackage{graphicx}          

\usepackage{amsmath}
\usepackage{color}

\def\changes#1{{\color{black}#1}}
\def\chg#1{{\color{black}#1}}

\usepackage{amssymb}
\usepackage{amsthm}
\usepackage{dsfont}
\usepackage{enumerate}
\usepackage[top=2.5cm, bottom=3cm, left=3cm , right=3cm]{geometry}

\newtheorem{thm}{Theorem}
\newtheorem{cor}[thm]{Corollary}
\newtheorem{conject}[thm]{Conjecture}
\newtheorem{defn}{Definition}
\newtheorem{prop}[thm]{Proposition}
\newtheorem{remark}{Remark}

\newcommand\mb[1]{\mbox{#1}}
\newcommand{\G}{\mathcal{G}}

\newcommand{\V}{\mathcal{V}}
\newcommand{\E}{\mathcal{E}}

\newcommand{\graph}{ \G = \left( \V , \E \right) }
\newcommand{\id}{\mathds{I}}

\newcommand{\1}{\mathbf{1}}

\newcommand{\footnoteremember}[2]{\footnote{#2}
\newcounter{#1}
\setcounter{#1}{\value{footnote}}}
\newcommand{\footnoterecall}[1]{\footnotemark[\value{#1}]}

\title{Graph diameter, eigenvalues, and minimum-time consensus\footnote{The authors would like to thank P. Van Dooren for the enlightening discussions they had with him. This paper presents research
results of the Belgian Network DYSCO (Dynamical Systems, Control, and Optimization), funded by the Interuniversity
Attraction Poles Programme, initiated by the Belgian State, Science Policy Office, \chg{and the Concerted
Research Action (ARC) of the French Community of Belgium.} \chg{This work has benefitted from many discussions within the FP7 NoE HYCON2.} The scientific responsibility rests
with its authors. R. M. J. is a F.R.S.-FNRS Research Associate, G. V. is a F.R.I.A. Fellow.}}
\author{J. M. Hendrickx\footnoteremember{ICTEAM}{ICTEAM Institute at the Universit\'e catholique de Louvain, Avenue Georges
Lemaitre 4, B-1348 Louvain-la-Neuve, Belgium; \{julien.hendrickx, raphael.jungers, guillaume.vankeerberghen\}@uclouvain.be.}, 
R. M. Jungers\footnoterecall{ICTEAM} , 
A. Olshevsky\footnoteremember{dont}{Department of Industrial and Enterprise Systems Engineering, University of Illinois at Urbana-Champaign, Urbana, IL, 61801, USA; aolshev2@illinois.edu.}, 
G. Vankeerberghen\footnoterecall{ICTEAM} }

\begin{document}

\maketitle

\begin{abstract}  We consider the problem of achieving average consensus \changes{in the minimum number of linear iterations} on a fixed, undirected
graph. We are motivated \changes{by the task of deriving lower bounds for consensus protocols} and by the so-called ``definitive consensus conjecture'' which states that for an undirected connected graph $\G$ with
diameter $D$ there exist $D$ matrices whose nonzero-pattern complies with the edges in $\G$ and whose product equals the all-ones matrix. 
Our first result is a counterexample to the definitive consensus conjecture, \changes{which is the first improvement of the diameter lower bound for linear
consensus protocols}.  We then provide some algebraic conditions under which this conjecture holds, which we use to establish that all distance-regular graphs satisfy the definitive consensus conjecture.  
\end{abstract}
\section{Introduction}

Consensus algorithms are a class of iterative update schemes commonly used as building blocks for distributed \chg{estimation and}
control protocols. The \chg{progresses} made within recent years in analyzing and designing consensus algorithms have led to advances
in a number of fields, for example, \chg{distributed estimation and} inference in sensor networks \chg{\cite{ItalDKF,Xiao:2007cr,XiaoBoydSensor}},
distributed optimization \cite{4749425}, and distributed machine learning \cite{Ang:2008hc}. These are among the many
subjects that have benefitted from the use of consensus algorithms. \

One of the available methods to design an average consensus algorithm is to use constant update weights satisfying some conditions for convergence (as can be found in \cite{BlondelTsitsi} for instance). However, the associated rate of convergence might be a limiting factor, 
and this has spanned a literature dedicated to optimizing the speed of consensus algorithms. For example, 
\changes{in the discrete-time case in which we are interested,} recent work has studied optimizing the spectral gap of the stochastic update matrix \cite{Boyd:2004ve,Xiao200465} or choosing an optimal network structure \cite{Delvenne:2007qf}. 

Other recent work has focused on achieving average consensus in finite time. For instance,  \cite{ODCoUUD} shows that if the interaction network is given by a fixed undirected graph (\chg{about which the agents initially only know their neighbors and the number of nodes}) and if the agents have total recall of the previously sent and received information, then a control scheme that is optimal in time can be found. Again on fixed undirected networks, it is shown in \cite{4497790} that any node can compute any function of the initial values after using almost any constant weights for a certain finite number of iterations given that the agents have either the memory of their past values or one dedicated register of memory, \chg{and given that some parameters are set at design time based on the network and the weights to be used}. Moreover, \cite{4497790} presents a decentralized way of computing the required parameters if the nodes have sufficient memory, the computing power to check rank conditions on matrices, and know a bound on the total number of nodes. \chg{These results are further improved and extended in \cite{6161213}} by providing a decentralized procedure \chg{using} less iterations and \chg{not} requiring the agents to know a bound on the number of agents. However, this last fact implies that the procedure only works \chg{for} almost all initial conditions. \chg{Finally,} \cite{kibangou:hal-00735058} also treats the problem of computing the consensus value after a finite number of constant-weight steps when the agents have memory and computing power. It treats this problem in the context of decentralized Laplacian eigenvalues estimation and topology identification.

In this paper, we study \chg{finite-time} consensus in \chg{the setting} where nodes only memorize their current state and, as before, update it as a linear function of their state and the ones of their neighbors \chg{but, the arbitrary time-varying weights can be chosen at the time of design}. \changes{Then, the problem is to pick update matrices at design time, so that average consensus is achieved in finite-time. This setting has been investigated in \cite{ChiKai} with the added restriction to use left-stochastic matrices corresponding to different types of consensus algorithms, e.g., various forms of pairwise exchanges like gossip matrices. Algorithms and asymptotic results, mainly based on using spanning trees, are proposed there. \cite{5026/THESES} treats the same question when the matrices \chg{do not have} to be left-stochastic \chg{nor} have to correspond to a known consensus algorithm. It introduces a Gather and Distribute scheme for trees \cite[Section 4.2]{5026/THESES}, which is also similar to \cite[Algorithm 3.3]{ChiKai} or the ConvergeCast algorithm, that shows how to reach finite-time consensus on any tree or on any graph based on a spanning tree. \cite{5026/THESES} finally conjectures that on any connected undirected graph, it is possible to design consensus schemes which take only as long as the diameter of the underlying graph to achieve average consensus. }

\changes{The conjecture just cited is the starting point of our developments. That is,} \changes{we are concerned with picking update matrices,} \changes{at design time,} \changes{so that average consensus is achieved in the fastest possible time. Furthermore, we are concerned with deriving lower bounds on such consensus schemes.}

\changes{After posing the problem and presenting previous results in Section \ref{sec:problemsetting}, we answer the conjecture with a counterexample in Section \ref{sec:counterEx}.} \changes{This proves that on some graphs, consensus may not be achieved by linear consensus protocols  as fast as the diameter of the graph. We note that this is the first time the diameter lower bound on linear consensus protocols has been improved for any graph.} 

In Section \ref{sec:suffcond}, we provide some algebraic conditions under which the definitive consensus conjecture holds and show that these conditions \chg{are easily met} on all distance-regular graphs.  \changes{In particular, this implies that graphs which possess a lot of symmetry automatically satisfy the conjecture; we will make
this statement precise in that section}.

Finally, in Section \ref{last}, we show how our conditions can be used to find a number (possibly larger than the diameter) of matrices compliant with the graph that lead to average consensus and we introduce the notion of consensus number of a graph. \changes{Our results in Section \ref{last} are related with those obtained independently in \cite{Kibang,6315398}. In these references the author does not try to minimize the running time of the consensus protocols, but provides a general method for obtaining consensus in finite time. This method can be seen as a particular case of Corollary \ref{cor:beyondconj} in \chg{the present} paper and is further discussed at that point.} 

\section{Problem setting and previous results \label{sec:problemsetting}}
\subsection{Problem setting}
 \changes{For ease of exposition, we only consider undirected graphs and, when some results can be easily extended to digraphs, we explicitly mention it.}

Given a connected graph $\graph$, we define the \emph{distance},  $\delta(i,j)$, separating two nodes $i, j \in \V$ as the number of edges in a shortest-path between $i$ and $j$. Following this definition, the \emph{diameter}, $D(\G)$, of a graph corresponds to the maximum distance between two nodes, i.e., $D(\G) = \max_{i,j \in \V} \delta(i,j)$.
Another linked parameter is the \emph{eccentricity}, $\epsilon(i)$, of a node $i$. This is the maximum distance separating $i$ from another node, i.e., $\epsilon(i) = \max_{j \in \V} \delta(i,j)$.
The \emph{radius} $R(\G)$ of a graph is the minimum eccentricity over all nodes, i.e., $ R(\G) = \min_{i \in \V} \epsilon(i) = \min_{i \in \V} \left( \max_{j \in \V} \delta(i,j) \right)$.
Finally, a node is called \emph{central} if its eccentricity equals the radius of the graph. \chg{More details on these notions can be found in \cite{WestGraph}.}

Our setting is that of a fixed interaction network encoded by an undirected graph $\graph$. Agents are allowed to store only their current state and to synchronously update it to a linear combination of their state and the states of their neigbors. Hence, if $x_{i}^{(t)}$ denotes the state of agent $i$ at time $t$ and $\mathcal{N}_{i} \subset \V$ denotes the set of agents with which $i$ can communicate then the updated state of agent $i$ is $x^{(t+1)}_{i} = a_{ii}^{(t+1)}x^{(t)}_{i} + \sum_{j \in \mathcal{N}_{i}} a_{ij}^{(t+1)} x_{j}^{(t)}$ for some choice of weights $a^{(t)}_{ij}$. More compactly,  a synchronous linear update can be written as $\mathbf{x}^{(t+1)} = A^{(t+1)}\mathbf{x}^{(t)}$ where the states of the agents at time $t$ are in the column vector $\mathbf{x}^{(t)}$ and the weights $a_{ij}^{(t+1)}$ are the entries of $A^{(t+1)}$. The matrices $A^{(t)}$ are required to comply with the underlying graph $\graph$ in the
following sense. 

\begin{defn}
Given a graph $\graph$ on $N$ nodes, we define \changes{the set of matrices that comply with $\G$ as }
$\mathcal{M}(\G) = \left\{ A = [a_{ij}] \in \mathds{R}^{N \times N} \left|   i\neq j \, \, \mathrm{and} \,\, (i,j) \notin \E \right. \right.
 \left. \Rightarrow a_{ij} = 0  \right\}$.
 \end{defn}
We say that a network of agents has reached \emph{average consensus} at time $t^{*}$ if the state of each agent is equal to the average of the initial states, i.e., if $\mathbf{x}^{(t^*)} = (\frac{1}{N} \1^{T} \mathbf{x}^{(0))}) \1 =  \overline{\mathbf{x}^{(0)}} \1$ with $\1$ the column vector of ones. This is the case for any initial vector $\mathbf{x}^{(0)}$ if and only if $A^{(t^*)}\cdots A^{(1)} = {1\over N} \1 \1^{T}$.

In Sections \ref{sec:counterEx} and \ref{sec:suffcond} we answer and analyze the following conjecture.
\begin{conject}[definitive consensus conjecture, \cite{5026/THESES}]\label{defconsconj2} 
For any connected graph $\G$ on $N$ vertices, there exist $D(\G)$ matrices $\mb{A}^{(t)}$ that comply with the graph and are such that 
\begin{eqnarray} \label{eq:prodEq} \mb{A}^{\left( D(\G) \right)} \mb{A}^{ \left( D(\G)-1 \right)} \cdots \mb{A}^{(1)} = {1 \over N} \1 \1^{T} \, , \end{eqnarray}
so that $\mb{A}^{\left( D(\G) \right)} \mb{A}^{ \left( D(\G)-1 \right)} \cdots \mb{A}^{(1)} \mathbf{x}^{(0)} = \overline{\mathbf{x}^{(0)}} \1$. 
\end{conject}
In other words, the conjecture states that it is always possible to reach average consensus \changes{with linear updates} in only $D(\G)$ steps.

We remark that the definitive consensus conjecture may be viewed as a statement \chg{on} the feasibility of certain polynomial
equalities. Indeed, developing Equation (\ref{eq:prodEq}) leads to a system of $N^{2}$ polynomial 
equations in the $D(\G)(N + 2|\E|)$ weights.
For instance, \chg{let us consider a simple graph of diameter $2$} such as $\mb{P}_{3}$ pictured on Fig. \ref{fig:path3}. In $\mb{P}_{3}$ there is no link between nodes $1$ and $3$, hence $a_{13}^{(t)}$ and $a_{31}^{(t)}$ must be zero. Thus, we have the liberty to choose the other 14 nonzero entries in order to fulfil 
\begin{eqnarray} \label{eq:p3}
\begin{pmatrix} a_{11}^{(2)} & a_{12}^{(2)} & 0 \\ a_{21}^{(2)} & a_{22}^{(2)} & a_{23}^{(2)} \\ 0 & a_{32}^{(2)} & a_{33}^{(2)} \end{pmatrix}  \begin{pmatrix} a_{11}^{(1)} & a_{12}^{(1)} & 0 \\ a_{21}^{(1)} & a_{22}^{(1)} & a_{23}^{(1)} \\ 0 & a_{32}^{(1)} & a_{33}^{(1)} \end{pmatrix}= {1 \over 3} \begin{pmatrix} 1 & 1 &1\\1 & 1 & 1\\1 & 1 &1 \end{pmatrix} \, . \nonumber
\end{eqnarray}
Developing the product on the left-hand side in this last equation yields a system of 9 polynomial equations in the 14 weights. In this system, the monomials in the polynomial equation of an entry $ij$ correspond to all the walks of length $D(\G)$ from $j$ to $i$.

\begin{figure}[h]
\begin{center}
\includegraphics[width=0.4\textwidth]{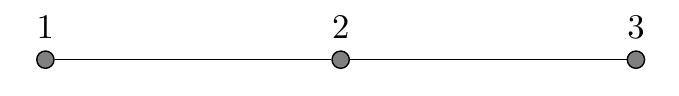}
\caption{\label{fig:path3} $\mb{P}_{3}$, the path graph on $3$ vertices. }
\end{center}
\end{figure}


\subsection{Previous results \label{sec:previous}}
Different related schemes based on trees have been proposed \cite[Algorithm 3.3]{ChiKai}, \cite[Section 4.2]{5026/THESES} or the ConvergeCast algorithm.  \cite{5026/THESES} showed that these schemes take only $D(\mathcal{T})$ steps on any tree $\mathcal{T}$ hence, trees satisfy the definitive consensus conjecture. 

The fact that the conjecture holds on trees induces that one can reach average consensus on any \chg{connected} graph with a number of updates equal to at most two times the radius \cite{5026/THESES}. Indeed, a spanning tree generated by Breadth First Search from a central node has a diameter of at most two times the radius of the original graph. Moreover, the definitive consensus conjecture holds on graphs admitting a \emph{Diameter Preserving Spanning Tree} (DPST),  which is a tree on the same set of nodes, using a subset of the edges and having the same diameter as the original graph. \chg{One can find} a characterization of graphs admitting a DPST in \cite{Buckley:1988fk}.

When the Gather and Distribute scheme cannot be applied in $D(\G)$ steps, one is left with a large system of polynomial equations in the weights. Such systems are not easy to solve. Gr\"{o}bner basis theory and Buchberger's algorithm are one way, but due to its exponential time complexity this algorithm becomes intractable even on small graphs. Some heuristics have also been proposed in \cite{5026/THESES}.

\section{\label{sec:counterEx}\chg{The diameter lower bound is not always tight}}
We now demonstrate that the conjecture is false using the diameter-2 graph $\G_{cx}$ pictured on Fig. \ref{fig1}.
\begin{figure}[h!]
\begin{center}
\includegraphics[height=3.5cm]{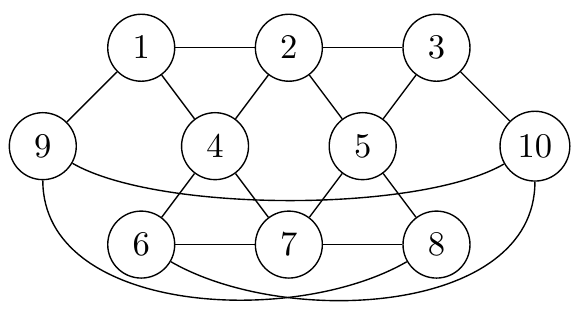}    
\caption{$\G_{cx}$, the counterexample.}  
\label{fig1}                                
\end{center}                             
\end{figure}

\begin{prop}
\chg{The graph $\G_{cx}$ has diameter 2 but} no two matrices $A, B$ compliant with $\G_{cx}$ can satisfy $AB = {1 \over N} \1 \1^{T}$.
\end{prop}
\begin{proof} \changes{We let the reader verify that $\G_{cx}$ has diameter 2 and} proceed by contradiction. Suppose that there exist two ten by ten matrices $A = [a_{ij}], \, B = [b_{ij}]$ in $\mathcal{M}(\G_{cx})$ such that $AB = \1\1^{T}$ (the coefficient ${1 \over N}$ has been dropped here as any coefficient can be obtained by scaling the matrices).
An entry $[AB]_{ij}$ corresponds to the sum, over all walks of length two, of the product of the weights chronologically set on the crossed edges. For instance, there is only one path of length two from $1$ to $6$ and it passes through $4$. The same holds from $1$ to $7$ so these entries in $AB$ read as $\left[ AB \right]_{61} = a_{64} b_{41} = 1$ and $\left[ AB \right]_{71} = a_{74} b_{41} = 1$.
These two relations can hold only if $a_{64} = a_{74}$.
Similarly, one finds a corresponding relation by looking at the paths from $3$ to $7$ and  from $3$ to $8$: $a_{75} = a_{85}$. Then let us look at $[AB]_{62}$, we have $[AB]_{62} = a_{64}b_{42} = 1$ so $b_{42} = {1 \over a_{64}} = {1 \over a_{74} }$. Similarly we find $b_{52} = {1 \over a_{85}} = {1 \over a_{75}}$. Now, there are two paths of length two from $2$ to $7$, one through $4$ and one through $5$. Therefore, \chg{and by using the equalities obtained we get ${ [AB]_{72} = a_{74}b_{42} + a_{75}b_{52} = 1 + 1}$ which contradicts $AB = \1 \1^{T}$.} 
\end{proof}
\chg{Hence, $\G_{cx}$ is a graph of diameter 2 on which on which it is impossible to linearly reach average consensus in two steps.} Nevertheless, there are many graphs on which \chg{a number of steps equal to the diameter is enough} as we will see in the next section.

\section{\label{sec:suffcond}\chg{Graphs for which the diameter lower bound is tight}}
In this section we demonstrate that when there exists one matrix in $\mathcal{M}(\G)$ that satisfies some particular conditions \chg{then one can easily find $D(\G)$ matrices that comply with $\G$ and lead to average consensus}. We then give such a matrix for path graphs and distance-regular graphs and we illustrate these results on a distance-regular graph with 18 nodes.

\subsection{\chg{Solutions based} on one matrix compliant with $\G$}
We first state and prove a general theorem that highlights the core conditions needed on the matrix. Then, we prove a corollary that is \chg{more actionable}. 
Let us recall that the \emph{minimal polynomial} of a matrix is the least degree monic polynomial $p$ such that $p(M) = 0$; we denote it by $m_{M}(x)$ for a particular matrix $M$.
Cayley-Hamilton's Theorem shows that the minimal polynomial has the same roots as the characteristic polynomial, i.e., the eigenvalues, but with multiplicity equal to the largest size of a Jordan 
block corresponding to the eigenvalue.

\begin{thm}\label{theo:general}
Let $\G$ be a connected graph with diameter $D$. If there exists a matrix $M \in \mathcal{M}(\G)$ such that the three following conditions are satisfied,
\begin{enumerate}[(a)]
\item $\exists k \in \mathds{R}: {\mathrm{ker}(M-k\id) = \mathrm{ker}(M^{T} - k \id) =   \left\{ \alpha \1 | \alpha \in \mathds{R} \right\}}$; 
\item $m_{M}(x)$ has degree $D+1$;
\item $M$ has real eigenvalues, i.e., the roots $k, \lambda_{1}, \dots, \lambda_{D}$ of $m_{M}(x)$ are real;
\end{enumerate}
then the $D$ matrices $M - \lambda_{t}\id$ are in $\mathcal{M}(\G)$ and their product in any order is equal to a multiple of $\1 \1^{T}$. \\
\end{thm}
\begin{proof}
By assumption $(a)$ and Cayley-Hamilton's Theorem, $(x-k)$ is a factor of $m_{M}(x)$. So, let $q(x) = { m_{M}(x) \over (x-k)}$. By definition of the minimal polynomial
$ m_{M}(M) = (M-k \id) q(M) = 0$.
Again by assumption $(a)$, this relation holds only if each column of $q(M)$ is a multiple of $\1$,
\begin{eqnarray} \label{eq:onescolumns}
 q(M) = \biggl[ \alpha_{1}\1 \,\, \alpha_{2} \1 \,\, \dots \,\, \alpha_{N} \1  \biggl] \, . 
 \end{eqnarray}
We now prove that the $\alpha_{i}$'s must be equal. Since $\1^{T}$ is a left-eigenvector of $M$, $\1^{T} p(M) = p(k) \1^{T}$ for any polynomial $p$. So, if we pre-multiply each side in (\ref{eq:onescolumns}) by $\1^{T}$ we get $q(k) \1^{T} = N (\alpha_{1} \,\, \alpha_{2} \,\, \cdots \,\, \alpha_{N} )$.
Thus, $ \alpha_{1} = \dots =  \alpha_{N} = { q(k ) \over N}$ and (\ref{eq:onescolumns}) can be written $q(M) = { q(k) \over N} \1 \1^{T}$. 
By assumptions $(b)$ and $(c)$, $q(M)$ can be factorized in $D$ real linear factors which taken on $M$ give the matrices $A^{(t)} \triangleq (M - \lambda_{t} \id) \, , t=1,\dots,D $.
These matrices are indeed compliant with $\G$ if $M$ is. Note that $q(k)\neq0$ as (\ref{eq:onescolumns}) would otherwise violate the definition of the minimal polynomial, so we can scale by $1/q(k)$ to obtain the claimed result. Finally, the order in which one applies the matrices has no importance as they commute. 
\end{proof}

Theorem \ref{theo:general} has a simple consequence if $M$ is chosen symmetric and nonnegative irreducible.

\begin{cor} \label{cor:suffcond}
Let $\G$ be a connected graph with diameter $D$. If there \changes{exists} a symmetric irreducible nonnegative matrix $M \in \mathcal{M}(\G)$ satisfying the two conditions
\begin{enumerate}[(I)]
\item $M\1 = k\1$ for some $k \in \mathds{R}$;
\item $M$ has only $D+1$ distinct eigenvalues,
\end{enumerate}
then there exist $D$ real matrices $A^{(t)}$ that comply with the graph and are such that ${A^{(D)}A^{(D-1)} \cdots A^{(1)} = {1 \over N} \1\1^{T}}$.
\end{cor}
\begin{proof}
We check that all the conditions in Theorem \ref{theo:general} are met. As $M \1 = k \1$ and $M = M^{T}$, $\1$ is in $\mathrm{ker}(M-k \id)$ and in $\mathrm{ker}(M^{T}- k\id)$. The fact that these eigenspaces have dimension one follows from Perron-Frobenius' Theorem since $\1$ is a positive eigenvector, and $M$ is nonegative irreducible. Hence, condition $(a)$ is met.
Now we look at conditions $(b)$ and $(c)$.The matrix $M$ is diagonalizable since real and symmetric. Thus, all the Jordan blocks in its Jordan normal form have dimension one. So by Cayley-Hamilton, the degree of the minimal polynomial equals the number of distinct eigenvalues. Finally, the eigenvalues are real because $M$ is symmetric. 
\end{proof}
The main challenge in finding a matrix in $\mathcal{M}(\G)$ that fulfils the hypotheses of Corollary \ref{cor:suffcond} is condition $(II)$. The other conditions can easily be met with a matrix in $\mathcal{M}(\G)$. For instance, if $A$ is the adjacency matrix of a connected graph then $M = A - \mathrm{diag}( A\1 ) + \max(A\1) \id $ is symmetric nonnegative and irreducible. Indeed, a nonnegative matrix $B \in \mathds{R}^{N\times N}$ is irreducible if and only if the digraph on $N$ nodes with arc set $\left\{ (i,j) \left|  \, [B]_{ji} >0 \right. \right\}$ is strongly connected. The digraph corresponding to the nonzero entries of this particular $M$, and any matrix that has the same nonzero entries as $A$, is a directed version of $\G$ which is strongly connected because $\G$ is connected. One can easily check that $M = A - \mathrm{diag}( A\1 ) + \max(A\1) \id $ also has $\1$ as eigenvector.

We end this section with a Proposition that sheds some light on conditions $(b)$ and $(II)$ and might help in proving that they are satisfied. Let us first define the \emph{algebra generated by a matrix}; $\mathcal{A}_{M}  =  \left\{ p(M) \in \mathds{R}^{N\times N} \left| \, p \in \mathds{R}[x] \right. \right\}$ where $\mathds{R}[x]$ denotes the set of univariate polynomials with real coefficients. This set $\mathcal{A}_{M}$ is an algebra because it has the structure of a real vector space with a bilinear law of composition between vectors. Indeed, the powers of $M$ can be seen as vectors and the matrix multiplication yields a bilinear product from $\mathcal{A}_{M} \times \mathcal{A}_{M}$ to $\mathcal{A}_{M}$. Note that by definition, the degree $d$ of the minimal polynomial $m_{M}(x)$ is the dimension of $\mathcal{A}_{M}$. Indeed, it gives the first power $M^{d}$ that can be generated by a linear combination of the powers $0,\dots,d-1$. Many \changes{interesting} proofs in algebraic graph theory use the algebra generated by the adjacency matrix of a graph. The following Proposition \changes{is a straightforward generalization} of one such result to the matrices that comply with a graph or a digraph, \changes{its proof is included for the sake of completeness.} 
\begin{prop} \label{lem:dimalg}
For any matrix $M \in \mathcal{M}(\G)$ such that there exist $p \in \mathds{R}[x]$ with $p(M)$ having no zero entry it holds that $\mathrm{dim}(\mathcal{A}_{M}) \geq D(\G) + 1$.
\end{prop} 
\begin{proof}
Let $p \in \mathds{R}[x]$ be a least degree monic polynomial such that $p(M)$ has no zero entry, denote its degree by $d$. We claim that $d \geq D(\G)$.
Indeed, $d$ cannot be less than the diameter since \changes{if node $j$ is at distance $D(\G)$ from $i$ then} $[M^{k}]_{ij}$ is zero for  $k<D(\G)$. Finally, as a linear combination of $\{I, M, \ldots, M^{d}\}$ is needed to have a matrix with no zero entries, one can show by contradiction that the powers $0$ to $d$ must be linearly independent, so $\mathrm{dim}(\mathcal{A}_{M}) \geq d+1\geq D(\G)+1$.   
\end{proof}
Proposition \ref{lem:dimalg} shows that conditions $(b)$ and $(II)$ correspond to a lower bound on the degree of the minimal polynomial of a matrix that satisfies the other conditions.
\changes{\begin{remark}
Theorem \ref{theo:general} and Proposition \ref{lem:dimalg} are also valid on digraphs if one adapts the definitions of $\mathcal{M}(\G)$ to arcs and the notion of diameter. For instance, without going into details, the adjacency matrix of a De Bruin digraph is a matrix that satisfies Theorem  \ref{theo:general} for this digraph.
\end{remark}}
\subsection{Paths and distance-regular graphs}
\chg{We now apply Corollary \ref{cor:suffcond} on path graphs with the following Proposition}. Note that the schemes discussed in Subsection \ref{sec:previous} already provide weights leading to average consensus on a path network. Nevertheless, this Proposition shows a different and arguably more decentralized way to reach average consensus on such a network.

\begin{prop} \label{prop:paths}
Let $\mb{P}_{N}$ be the path graph on $N$ nodes with adjacency matrix $A_{P_{N}}$ such that nodes $1$ and $N$ have degree one. Then $M = A_{P_{N}} + \mathrm{diag}\{1,0,\dots,0,1\} $
has $N$ real distinct eigenvalues $2,\lambda_{1},\dots,\lambda_{N-1}$ and ${(M-\lambda_{N-1}\id) \cdots (M - \lambda_{1} \id) = \1 \1^{T} }$.
\end{prop}
\begin{proof} We prove that $M$ satisfies the hypotheses of Corollary \ref{cor:suffcond}. It is readily seen that $M$ is nonnegative, symmetric, in $\mathcal{M}(\mb{P}_{N})$ and such that $M\1 = 2 \1$. Moreover, $M$ is irreducible as the digraph on $N$ nodes with arc set $\left\{ (i,j) | \, [M]_{ji} >0 \right\}$ corresponds to the directed version of $\mb{P}_{N}$ which is strongly connected. It remains to prove that $M$ has $D(P_{N}) + 1$ distinct eigenvalues which we do by applying Proposition \ref{lem:dimalg}. The sum of the powers $0$ to $N$ of $A_{P_{N}}$ has no zero entries because any two nodes are at most at distance $N$. Hence, given the definition of $M$, $(M+\id)^{N}$ has no zero entries. Thus, by Proposition \ref{lem:dimalg}, $D(\mb{P}_{N}) + 1 \leq \mathrm{dim}(\mathcal{A}_{M})$. On the other hand, $D(\mb{P}_{N}) = N-1$ and Cayley-Hamilton's Theorem implies that $\mathrm{dim}(\mathcal{A}_{M}) \leq N$ hence, $\mathrm{dim}(\mathcal{A}_{M}) = N$. Then, by Corollary \ref{cor:suffcond}, 
$ (M-\lambda_{N-1}\id) \cdots (M - \lambda_{1} \id) = { \alpha \over N} \1\1^{T}$. 
Finally, one can show that ${ \alpha \over N}$ must be equal to $1$ by computing the entry $\left[(M-\lambda_{N-1}\id) \cdots (M - \lambda_{1} \id) \right]_{1N}$. 
\end{proof} 

We can also directly apply Corollary \ref{cor:suffcond} on a large class of graphs called distance-regular graphs. We refer the reader to \cite[Chapter 20]{Biggs:1994kx} for a clear introduction to this class. For instance, the cycles, the Petersen graph, the hypercube graphs and the complete bipartite graphs are distance-regular. Many other examples of such graphs can be found in \cite{WolframDistReg}. For the sake of completeness we provide their definition in which $\G_{r}(i)$ denotes the set of nodes at distance $r$ from node $i$ in $\G$.
\begin{defn} \label{def:distreg}
A  connected graph $\G$ with diameter $D$ is \emph{distance-regular} if there exist integers \\
$\left\{ b_{0},b_{1},\dots,b_{D-1}; c_{1}, c_{2},\dots,c_{D} \right\}$
such that for any pair $i,$ $j$ of $r$-distant vertices we have
\begin{enumerate}
\item  for $1\leq r \leq D$, $c_{r}$ is the number of vertices in $\G_{r-1}(j)$ that are adjacent to $i$;
\item  for $0 \leq r \leq D-1$, $b_{r}$ is the number of vertices in $\G_{r+1}(j)$ that are adjacent to $i$. 
\end{enumerate}
\end{defn}
From the definition, one can see that all distance-regular graphs are regular of degree $b_{0}$, which will be used in proving Proposition \ref{prop:distreg}.

\begin{prop}\label{prop:distreg}
Let $\G$ be a distance-regular graph, then there exist $D(\G)$ matrices in $\mathcal{M}(\G)$ whose product equals a multiple of $\1\1^{T}$.
\end{prop}
\begin{proof}
\chg{Theorem 20.7 in \cite{Biggs:1994kx} implies that the adjacency matrix $A$ of any distance-regular graph $\G$ has $D(\G)+1$ distinct eigenvalues (this result due to R. M. Damerell uses the algebra generated by the powers of $A$). It is readily seen that} $A$ is nonnegative, symmetric and irreducible. Moreover, as any distance-regular graph is $b_{0}$-regular, we have $A\1 = b_{0} \1$. Thus, Corollary \ref{cor:suffcond} is applicable on any such graph simply with $M = A$. 
\end{proof}
As an example, let us consider the Pappus graph pictured on Fig. \ref{fig:pappus}. 
\begin{figure}[h!]
\begin{center}
\includegraphics[height=3.5cm]{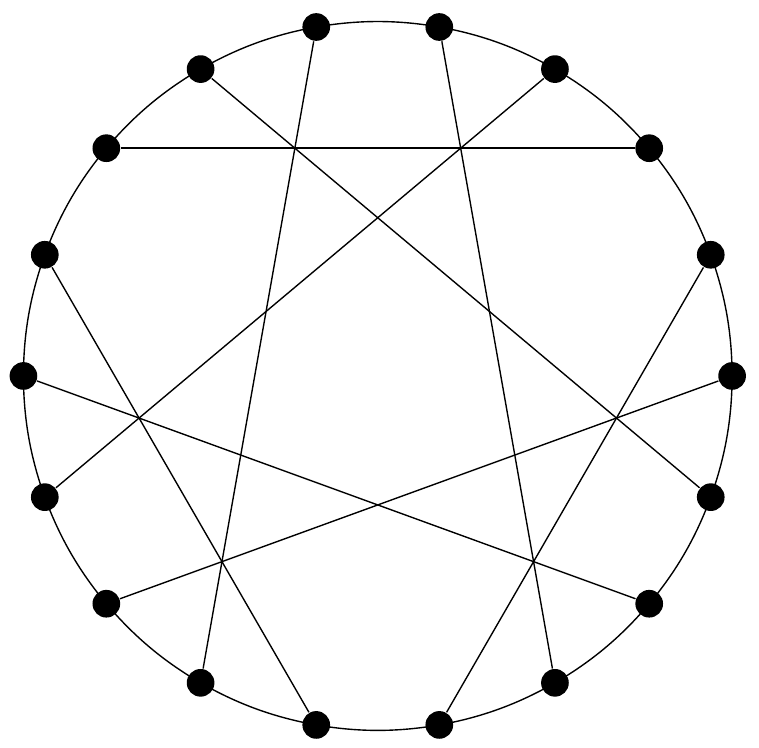}    
\caption{The Pappus graph.} 
\label{fig:pappus}                            
\end{center}                                 
\end{figure}
It has $18$ nodes, diameter $4$ and is distance-regular with integers $\{3,2,2,1; \, 1,1,2,3 \}$ \cite{WolframDistReg}. Its adjacency matrix $A$ has the spectrum (with multiplicities) $\{ -3^{(1)}, -\sqrt{3}^{(6)}, 0^{(4)}, \sqrt{3}^{(6)}, 3^{(1)} \}$. So, the four weight matrices 
\begin{equation}
\begin{array}{lr} 
A^{(1)} = {1 \over 6}\left(A + 3 \id\right),\, & A^{(2)} = {1 \over 3 + \sqrt{3}}\left(A + \sqrt{3} \id\right), \\
\,A^{(3)} = {1 \over 3}A,\, & A^{(4)}  =  {1 \over 3 - \sqrt{3}}\left(A - \sqrt{3} \id \right),
\end{array} \nonumber
\end{equation}
lead the Pappus network to average consensus. The evolution of $18$ states according to these weights and starting from an $\mathbf{x}^{(0)}$ uniformly distributed in $[-1.5,1.5]$ is illustrated on Fig. \ref{fig:pappusconv}.
\begin{figure}[h!]
\begin{center}
\includegraphics[width=0.8\textwidth]{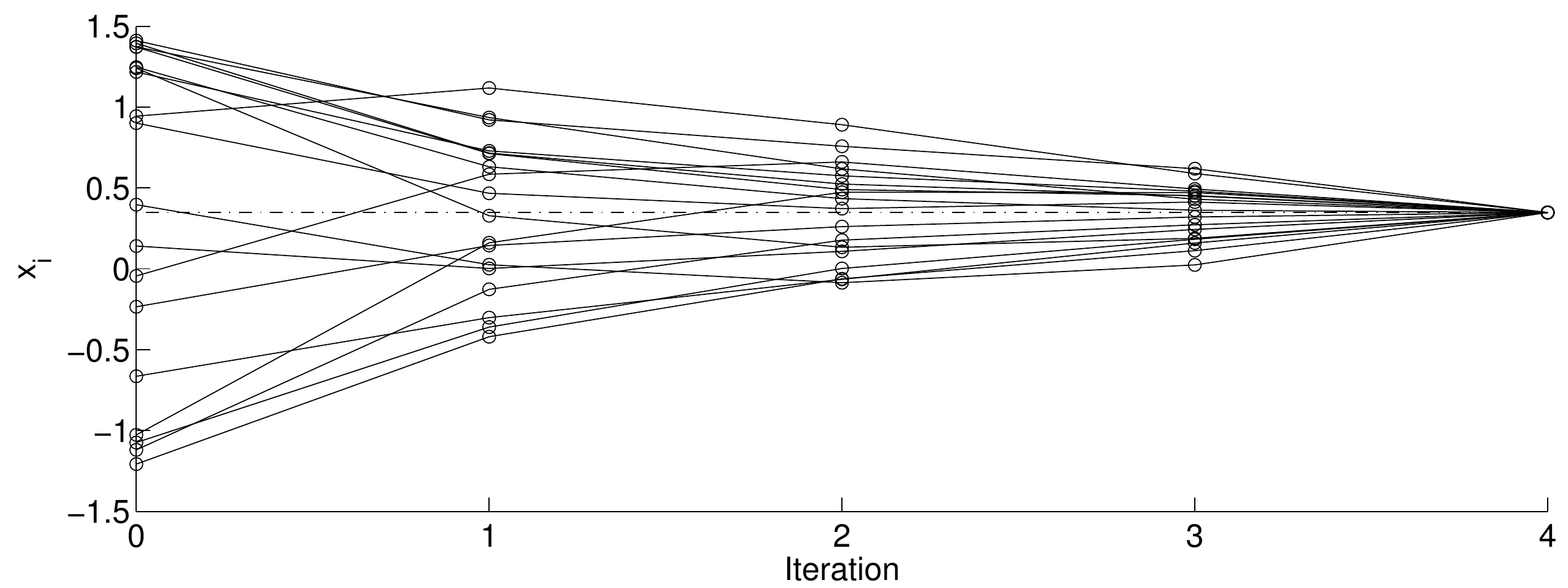}    
\caption{Evolution of $18$ uniformly distributed initial states on the Pappus graph using the given weights (circles linked by plain lines). The dashed line has height equal to the average of the initial states.}
\label{fig:pappusconv}                                 
\end{center}                                
\end{figure}

\changes{Finally, we remark on an interpretation of the last proposition, namely that undirected, connected graphs \changes{with symmetry among pairs of nodes at a particular distance from each other} satisfy the conjecture. \changes{For instance}, one notion of a highly symmetric graph is {\em distance transitivity}: given any four nodes $a,b$ and $u,v$ such that the distance from $a$ to $b$ is the 
same as the distance from $u$ to $v$, there is a graph automorphism\footnote{An automorphism of a graph is a relabeling of 
vertices that preserves the presence of edges. Formally, it is a map $\phi$ from the vertex set to itself such that $(i,j)$ is an edge in the
graph if and only if $(\phi(i), \phi(j))$ is. The set of all automorphisms form a group whose size is a measure of the 
symmetry of the graph.} which carries $a$ to $u$ and $b$ to $v$. Intuitively, the group of automorphisms of a distance-transitive
graph is rich enough so that any pair of vertices can be mapped by a suitable automorphism to every other pair at the same distance. We 
refer the reader to \cite{Biggs:1994kx} for more background and examples of these graphs.}

\changes{ \changes{It can be seen that} all distance-transitive graphs are distance-regular \changes{\cite{Biggs:1994kx}}, and consequently all distance-transitive graphs satisfy the definitive consensus conjecture by the last proposition}. 

\section{Finite-time linear average consensus and the consensus number of a graph\label{last}}
\chg{In the previous section we proved Corollary \ref{cor:suffcond} and used it to exhibit graphs on which $D(\G)$ compliant matrices leading to consensus can be found.} This same corollary might be adapted to easily find a finite number (maybe larger than the diameter) of matrices leading a fixed network of agents to average consensus. 

 \begin{cor}  \label{cor:beyondconj}
Let $\G$ be a connected graph, if there exists a symmetric irreducible nonnegative matrix $M \in \mathcal{M}(\G)$, with $\1$ as eigenvector and $s$ distinct eigenvalues
then, there exist $s-1$ real matrices $A^{(t)}$ that comply with the graph and are such that $A^{(s-1)} \cdots A^{(1)} = {1 \over N} \1\1^{T}$.
 \end{cor}
Proving this goes the same way as for Corollary \ref{cor:suffcond} but with $D(\G)$ replaced by $s$.
For instance, on the counterexample $\G_{cx}$, the matrix $M = A_{\G_{cx}} - \mathrm{diag}(A_{\G_{cx}} \1) + 4 \id $ fulfils the hypotheses of Corollary \ref{cor:beyondconj} with $M\1=4\1$ and $s=4$. 
Hence, average consensus is linearly achievable on $\G_{cx}$ in three steps, that is one more than its diameter.

This matrix $M \triangleq A - \mathrm{diag}( A\1 ) + \max(A\1) \id$ (where $A$ is the adjacency matrix of the graph) is the one that has been analyzed for finite-time average consensus in \cite{Kibang,6315398}. The author proves there that average consensus is always achievable in a number of steps equal to $s-1$ when $s$ is the number of distinct eigenvalues of  that particular $M$. His proof relies on joint diagonalization of matrices of the form $\alpha_{i} \id + \beta M$. On some graphs, this matrix does not achieve the minimum number of distinct eigenvalues. For instance, the graph on Fig. \ref{fig:double5cyc} is such that $A - \mathrm{diag}( A\1 ) + \max(A\1) \id$ has five distinct eigenvalues but we could find a matrix fulfilling Corollary \ref{cor:beyondconj} with only three distinct eigenvalues (which also equals its diameter plus one hence, this graph fulfills the conjecture).

\begin{figure}[!h]
\begin{center}
\includegraphics[scale=1]{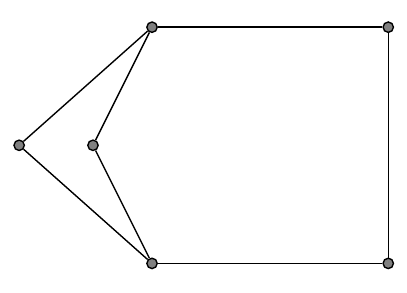}
\caption{\label{fig:double5cyc} A graph such that $A - \mathrm{diag}( A\1 ) + \max(A\1) \id$ does not achieve the minimum number of distinct eigenvalues over the set of symmetric nonnegative irreducible matrices in $\mathcal{M}(\G)$.  }
\end{center}
\end{figure}

The fact that the \chg{diameter lower bound is not always tight} raises the question of the least number of steps necessary to linearly achieve average consensus on a connected graph. We call this least integer the \emph{consensus number} of a graph. It is clear that the consensus number is lower bounded by the diameter of the graph but sometimes strictly bigger. Given the counterexample $\G_{cx}$ (Fig. \ref{fig1}) and the defining properties of distance-regular graphs, we might suspect that this parameter is linked to the number of edge-disjoint shortest-paths of maximum length. From Section \ref{sec:previous} we know that the consensus number is upper-bounded by $2R(\G)$.  Moreover, Corollary \ref{cor:beyondconj} shows that it is upper-bounded by $s-1$ for $s$ the number of distinct eigenvalues of any symmetric irreducible stochastic matrix in $\mathcal{M}(\G)$, which links our problem to the contemporary Inverse Eigenvalue Problem on a Graph \cite{1162.05333}. \changes{It is an open question \chg{whether} this graph invariant can be efficiently computed or if it can be linked to other graph invariants.}

\section{Conclusion} 
We have provided a counterexample to the definitive consensus conjecture. Thus, average consensus is not always achievable
\changes{with linear consensus protocols}  in a number of steps equal to the diameter of the graph, even though this number is sufficient for all the agents to influence each other. \changes{This is the first lower bound for linear consensus protocols on a graph which improves on the diameter of the graph}. However, we have shown that finding one matrix in $\mathcal{M}(\G)$ meeting some algebraic conditions could allow to easily find matrices leading to average consensus in finite time. These solutions are given by $D(\G)$ (respectively $s-1$) linear matrix factors if the matrix has a minimal polynomial of degree $D(\G)+1$ (respectively $s$). \changes{Since the adjacency matrix of a distance-regular graph has} a minimal polynomial of degree $D(\G)+1$, it follows that the definitive consensus conjecture holds on all distance-regular graphs. \changes{As mentioned earlier, this implies that any sufficiently symmetric graph in the sense of having an automorphism group rich enough to map any pair of vertices to any other pair at the same distance satisfies the definitive consensus conjecture}. Finally, we introduced the consensus number of a graph, which is the minimum number of iterations with which average consensus is linearly achievable on the graph. \changes{We believe that this parameter is linked to other graph-theoretic parameters and leave the question of its computability open.}

\changes{In this paper, we were concerned with finding matrices centrally at design-time. When nodes have enough memory and computing power, the decentralized methods for retrieving the Laplacian's spectrum presented in \cite{6315398} and the references therein could be adapted to let the nodes decentrally adapt their weights for future finite-time consensus iterations. }  

\newpage
\bibliographystyle{plain}        
\bibliography{Bibli_Paper} 

\end{document}